\documentclass[12 pt,reqno,oneside]{amsart}
\usepackage{amsmath,amsthm,amsfonts,amssymb}
\usepackage[latin1]{inputenc}
\usepackage[mathscr]{eucal}
\usepackage[usenames, dvipsnames]{color}
\usepackage{enumerate}
\usepackage{indentfirst}
\usepackage{a4wide}
\usepackage{graphicx}
\usepackage{url}

\usepackage{setspace} 

\theoremstyle{plain}
\newtheorem{teo}{Theorem}[section]
\newtheorem{cor}[teo]{Corollary}
\newtheorem{lem}[teo]{Lemma}
\newtheorem{prop}[teo]{Proposition}

\theoremstyle{definition}
\newtheorem{defn}[teo]{Definition}
\newtheorem{exa}[teo]{Example}
\newtheorem{rem}[teo]{Remark}

\numberwithin{equation}{section}

\def\bbP{{\mathbb P}}

\def\bbN{{\mathbb N}}

\def\bbT{{\mathbb T}}

\def\proc{\text{C}(\mathbb{T}^d;\mathcal{N},\mathcal{L})}
\def\procR{\text{C}(\mathbb{T}^d_+;\mathcal{N},\mathcal{L})}
\def\procp{\text{C}(\mathbb{T}^d; \mathcal{P}(\lambda), \text{Bin}(p))} 
\def\procpR{\text{C}(\mathbb{T}^d_+; \mathcal{P}(\lambda), \text{Bin}(p))}

\def\up{{U(\bbT^d; \mathcal{N},\mathcal{L})}}
\def\upP{{U(\bbT^d_+; \mathcal{N},\mathcal{L})}}
\def\upPd{{U(\bbT^{d+1}_+; \mathcal{N},\mathcal{L})}}

\def\procUp{\text{U}(\mathbb{T}^d_+; \mathcal{P}(\lambda), \text{Bin}(p))}
\def\procIp{\text{I}(\mathbb{T}^d_+; \mathcal{P}(\lambda), \text{Bin}(p))}

\def\upi{{I(\bbT^d; \mathcal{N},\mathcal{L})}}
\def\upPi{{I(\bbT^d_+; \mathcal{N},\mathcal{L})}}

\def\low{{L(\bbT^d; \mathcal{N},\mathcal{L})}}

\def\lowP{{L(\bbT^d_+; \mathcal{N},\mathcal{L})}}

\begin{document}

\baselineskip=26pt

\address[Valdivino~V.~Junior]
{Federal University of Goias
\\ Campus Samambaia, CEP 74001-970, Goi\^ania, GO, Brazil.}

\address[F\'abio~P.~Machado]
{Institute of Mathematics and Statistics
\\ University of S\~ao Paulo \\ Rua do Mat\~ao 1010, CEP
05508-090, S\~ao Paulo, SP, Brazil.}

\address[Alejandro Rold\'an-Correa]
{Instituto de Matem\'aticas, Universidad de Antioquia, Calle 67, no 53-108, Medellin, Colombia.}

\title[Uniform dispersion in growth models on homogeneous trees]{Uniform dispersion in growth models on homogeneous trees}
\author{Valdivino~V.~Junior}
\author{F\'abio~P.~Machado}
\author{Alejandro Rold\'an-Correa}

\noindent
\email{fmachado@ime.usp.br, alejandro.roldan@udea.edu.co, vvjunior@ufg.br}

\thanks{Research supported by FAPESP (2023/13453-5 and 2022/08948-2) and Universidad de Antioquia.}



\date{\today}

\begin{abstract}
We consider the dynamics of a population spatially structured in colonies that are vulnerable to catastrophic events occurring at random times, which randomly reduce their population size and compel survivors to disperse to neighboring areas. The dispersion behavior of survivors is critically significant for the survival of the entire species. In this paper, we consider an uniform dispersion scheme, where all possible survivor groupings are equally probable. The aim of the survivors is to establish new colonies, with individuals who settle in empty sites potentially initiating a new colony by themselves. However, all other individuals succumb to the catastrophe. We consider the number of dispersal options for surviving individuals in the aftermath of a catastrophe to be a fixed value $d$ within the neighborhood. In this context, we conceptualize the evolution of population dynamics occurring over a homogeneous tree. We investigate the conditions necessary for these populations to survive, presenting pertinent bounds for survival probability, the number of colonized vertices, the extent of dispersion within the population, and the mean time to extinction for the entire population.
\end{abstract}

\singlespacing

\maketitle

\section{Introduction}
\label{S: Introduction}
Some biological populations frequently face catastrophic events, such as epidemics and natural disasters, which can lead to the extinction of the species. Dispersal of individuals during catastrophes serves as a strategy to enhance a species' chances of survival. This strategy increases genetic diversity within separated populations, reduces intraspecific competition for resources, and helps individuals avoid predation or infections. Dispersal also allows for the colonization of new habitats less affected by the catastrophe, ultimately increasing the chance of long-term survival for at least a portion of the population. For further information on dispersal in the biological context, please refer to Ronce \cite{R2007}.\\

In \cite{JMR2016, JMR2020, JMR2022, MRV2018, MRS2015, S2014}, models have been introduced to examine various dispersal strategies in populations experiencing different types of catastrophes. These models aim to assess the influence of these strategies on population viability, contrasting them with scenarios where no dispersal occurs. Their analysis seeks to determine the best strategy (dispersion or no dispersion) by evaluating survival probability and population extinction times when a specific strategy is employed. In particular, in \cite{JMR2020, JMR2022, MRV2018, MRS2015}, the authors consider that when a colony of individuals suffers a catastrophe, the size of this population is reduced according to some probability law (binomial or geometric), and the surviving individuals disperse to neighboring sites to establish new colonies. New colonies can only form on empty sites.
Among the individuals that go to the same empty site, only one survives and the others die. Three dispersion schemes were considered, when each colony has $d$ neighboring sites:

\begin{itemize}
    \item  \textit{Optimal Dispersion:} When $r$ individuals survive a catastrophe, there are exactly $\min\{r,d\}$ successful attempts to establish new colonies.\\
    \item \textit{Independent Dispersion:} Each survivor picks a neighboring site at random and tries to create a new colony there.\\ 
    \item \textit{Uniform Dispersion:} For every $r$ survivors, consider all sets of numbers $r_1,\dots,r_{d} \in\bbN$ (\textit{occupancy set of numbers}) that are a solution to $r_1 + r_2 + \cdots + r_{d} = r.$ For each of these sets exactly one of the $r_i$ individuals (if $r_i>0$, and 0 otherwise) will succeed in colonizing the neighboring vertex i, $i=1, \dots, d$. Observe that $r_i$ may be 0. Here we consider that all possible survivor groupings are equally probable. So, the probability of having exactly $y \le \min\{r,d\}$ successful attempts when the number of survivors is $r$, is 
    \[ \frac{{{r-1}\choose{y-1}}}{{{d+r-1}\choose{r}}} {{d}\choose{y}}. \]
\end{itemize}
Taking everything into account, the population growth of each colony follows a birth and death process denoted as $\mathcal{N}$.
Each colony is associated with an independent exponential random time, with a mean of 1, indicating the occurrence time of a catastrophe. When a catastrophe occurs, it results in a reduction in the colony's size, following a specific probability distribution, $\mathcal{L}$. Although the paper is presented in a general context, several results detail specific cases. In particular, we examine the scenario where 
\begin{itemize}
    \item $\mathcal{N}$ follows a Poisson process with rate $\lambda$, denoted by $\mathcal{P}(\lambda)$, and
    \item $\mathcal{L}$ follows a binomial distribution whose parameters are the colony's size and $p$, denoted by $\text{Bin}(p).$
\end{itemize}

Next, consider the number of dispersal options for surviving individuals in the aftermath of a catastrophe to be a fixed value $d$ within the neighborhood. In this context, we conceptualize the evolution of population dynamics occurring over $\mathbb{T}^d$, a homogeneous tree where every
vertex has $d+1$ nearest neighbors, and over $\mathbb{T}^d_+$, a tree whose only difference from $\mathbb{T}^d$ is that
its origin has degree of $d$.

Individuals remaining after the catastrophe are  distributed among the nearest neighbor vertices, according to \textit{uniform dispersion}. Individuals that go to a vertex already occupied by a colony die. Among the individuals that go to the same (empty) vertex to create a new colony there, only one succeeds, the others die. Therefore, when a catastrophe occurs in a colony, that colony is replaced by $0, 1,\ldots$ or $d$ colonies, each colony is started by a single individual. We assume that initially all vertices of $\mathbb{G}$ are empty, except one called the \textit{origin}, which has a colony with a single individual. We denote this model by $\text{C}(\mathbb{G};\mathcal{N},\mathcal{L})$ where $\mathbb{G}$ is either $\mathbb{T}^d$ or $\mathbb{T}^d_+$.

This paper is divided into four sections. In Section 2, we introduce results on phase transition, survival probability, the number of colonies created in the model $\proc$. Besides that, we study the reach on the graph and the mean extinction time of the model $\procR$. In Section 3, we prove the results presented in Sections 2.
Finally, in Section 4, we contrast \textit{uniform dispersion} with \textit{independent dispersion}, by comparing our model with the model introduced by Machado et al.~\cite{MRV2018}.

\section{Main Results}

In order to establish the main results on  $\text{C}(\mathbb{T}^d;\mathcal{N},\mathcal{L})$
and $\text{C}(\mathbb{T}^d_+;\mathcal{N},\mathcal{L})$, we begin by introducing some definitions and notations.

Let $N_t$ be the number of individuals in a colony at time $t$ (after its creation) and $T$ be the time in which the catastrophe occurs in that colony. We denote by $N$ the number of individuals that survive the catastrophe before dispersing. In particular, for \linebreak $\text{C}(\mathbb{G}; \mathcal{P}(\lambda), \text{Bin}(p))$, the distribution of $N$ is given by
\begin{equation}\label{distparticular}
\mathbb{P}[N=n] =\left\{
\begin{array}{ll}
\dfrac{1-p}{\lambda p +1},&  n=0,\\ \\
\left (\dfrac{\lambda p}{\lambda p + 1} \right )^n \dfrac{\lambda +1}{\lambda (\lambda p +1)},& \, n \geq 1,
\end{array}
\right.
\end{equation}
and the probability generating  function is
\begin{equation}\label{fgpparticular}
\mathbb{E}[s^{N}]
= \frac{1-p(1-s)}{1+\lambda p(1-s)}.
\end{equation}

For more details, see Junior~\textit{et al}~\cite[Proof of Lemma 4.3]{JMR2016}.

Observe that  model $\text{C}(\mathbb{T}^d;\mathcal{N},\mathcal{L})$ is a continuous-time stochastic process with state space ${\mathbb{N}_0}^{\mathcal{V}(\mathbb{T}^d)}$, where $\mathcal{V}(\mathbb{T}^d)$ is the set of vertices of $\mathbb{T}^d$. The evolution of this process in time is denoted by $\eta_t$. For a vertex $x\in\mathcal{V}(\mathbb{T}^d)$, $\eta_t(x)$ is the number of individuals at time $t$ at vertex $x$.  We consider $|\eta_t|:=\sum_{x\in \mathcal{V}(\mathbb{T}^d)} \eta_{t}(x)$, the total number of individuals present in $\mathbb{T}^d$ at time $t$. Analogously for $\text{C}(\mathbb{T}^d_+;\mathcal{N},\mathcal{L})$.

\subsection{Phase Transition}

\begin{defn}
We say that the process ($\text{C}(\mathbb{T}^d;\mathcal{N},\mathcal{L})$ or $\text{C}(\mathbb{T}^d_+;\mathcal{N},\mathcal{L})$) \textit{survives}  if at every instant of time there is at least one alive colony.  
For $V_d$, the survival event, 
\[V_d = \{|\eta_t|>0, \text{ for all } t\geq0\}.\]
\end{defn}    
In the following result, we establish sufficient conditions for the probability of survival to be greater than or equal to zero.

\begin{teo}
\label{C: MCC2H}
Consider $\proc$ and let $V_d$ be the survival event. Then $\mathbb{P}(V_d) = 0 $ if
\begin{displaymath}
\mathbb{E} \left [ \frac{N}{N+d} \right] \leq \frac{1}{d+1}
\end{displaymath}
and
$\mathbb{P}(V_d) > 0 $ if
\begin{displaymath}
\mathbb{E} \left [ \frac{N}{N+d} \right]  > \frac{1}{d}.
\end{displaymath}
\end{teo}

\begin{cor}
\label{C: MCC2HP}
Consider $\procp$ and let $V_d$ be the survival event. Then  $\mathbb{P}(V_d) = 0 $ if
\[
\Phi \left( \frac{\lambda p }{\lambda p + 1}, 1,d+1\right ) \geq \frac{(\lambda p +1)[p(\lambda d + d + 1)-1]}{pd(d+1)(\lambda +1)}
\]
and
$\mathbb{P}(V_d) > 0 $ if
\[
\Phi \left( \frac{\lambda p }{\lambda p + 1}, 1,d+1\right )  < \frac{(\lambda p +1)[p(\lambda d + d -\lambda)-1]}{pd^2(\lambda +1)},
\]
where $\Phi(z,s,a)$ is the Lerch Transcendent Function given by 
\begin{align}
\Phi(z,s,a) = \sum_{j=0}^{\infty} \frac{z^j}{(a+j)^s}, \mid z \mid <1.
\end{align}
\end{cor}

In particular, the Lerch Transcendent Function satisfies 
\begin{align}
\Phi(z,1,a) =  \frac{1}{z^a} \left [ \ln \left (\frac{1}{1-z} \right) - \sum_{j=1}^{a-1}\frac{z^j}{j} \right ].
\end{align}
for every natural number $a$. 

\begin{exa} 
\label{E: Id}
Consider $\procp$ and let $p_c(d, \lambda)$ defined by
\[
p_c(d, \lambda) = \inf \{ p \geq 0; \mathbb{P}(V_d) > 0  \}.
\]

Then, for $\lambda=10$ and $d=30$, by using Corollary \ref{C: MCC2HP},
we have that 
\[ 0.0962 \leq p_c(30, 10) \leq 0.0996. \]

With independent dispersion scheme, using \cite[Theorem 3.2]{MRV2018} for the analogous parameter, we have
\[ 0.0936 \leq p_c(30, 10) \leq 0.0969.\] Due to the intersection of the ranges obtained from the rigorous results currently available, it is not possible to make a definitive comparison between the models with uniform dispersion and independent dispersion.
\end{exa}

\subsection{Probability of Survival}

\begin{teo}
\label{T: MCCHPE2}
Consider  $\proc$ and let $V_d$ be the survival event.  Then
\begin{displaymath}
\sum_{r=1}^{d+1} \left [(1 - \rho^r)\binom{d+1}{r}\sum_{n=r}^{\infty}\frac{\binom{n-1}{r-1}}{\binom{n+d}{d}}\mathbb{P}(N=n) \right] \leq \mathbb{P}(V_d) \leq 1 - \psi
\end{displaymath}
where $\psi$ and $\rho$ are, respectively, the smallest non-negative solutions of 
\begin{displaymath}
\sum_{y=1}^{d+1} \left[ s^y\binom{d+1}{y}\sum_{n=y}^{\infty}\frac{\binom{n-1}{y-1}}{\binom{n+d}{d}}\mathbb{P}(N=n) \right] = s - \mathbb{P}(N=0) \textrm { and }
\end{displaymath}
\begin{displaymath}
\sum_{y=0}^{d} \left [ s^y\binom{d}{y}\sum_{n=y}^{\infty}\frac{\binom{n}{y}}{\binom{n+d}{d}}\mathbb{P}(N=n)\right ] = s.
\end{displaymath}
\end{teo}

\begin{exa}
Consider C$(\mathbb{T}^{10}; \mathcal{P}(5), \text{Bin}(\frac{3}{5}))$. Then, 
\begin{displaymath}
\mathbb{P}(N = n ) = \left\{
                      \begin{array}{ll}
                       \frac{1}{10}, & \hbox{$n=0$;} \\
                        \frac{3}{10}\left(\frac{3}{4} \right )^n, & \hbox{$n \geq 1$.}                        
                      \end{array}
                    \right.
\end{displaymath}
Applying Theorem \ref {T: MCCHPE2} we have $\psi = 0.141484$ and $\rho = 0.162176$. So,
\begin{align*}
\frac{3}{10}\sum_{r=1}^{11} \left [(1 - 0.162176^r)\binom{11}{r}\sum_{n=r}^{\infty}\frac{\binom{n-1}{r-1}}{\binom{n+10}{10}} \left (\frac{3}{4} \right)^n \right] & \leq \mathbb{P}(V_{10}) \leq 1 - 0.141484 \\ 0.85153 \leq \mathbb{P}(V_{10}) \leq 0.858516.
\end{align*}
\end{exa}

\begin{teo}
\label{T: MCCHPE1L}
Consider $\proc$ and let $V_d$ be the survival event. We have that
\[ \lim_{d \to \infty} \mathbb{P}(V_d) = 1 - \nu
\]
where  $\nu$ is the smallest non-negative solution of $ \mathbb{E}(s^N) = s.$
\end{teo}

\begin{cor}
\label{C: MCC1HPS}
Consider $\procp$ and let $V_d$ be the survival event. Then
\[ \lim_{d \to \infty} \mathbb{P}(V_d) = \max \left \{ 0, \frac{p(\lambda +1) -1}{\lambda p} \right \}.
\]
\end{cor}

\subsection{The reach of the process}

We define $M_d$ as the reach of the process ($\text{C}(\mathbb{T}^d;\mathcal{N},\mathcal{L})$
or $\text{C}(\mathbb{T}^d_+;\mathcal{N},\mathcal{L})$), which, in simpler therms, represents the distance from the origin to the farthest vertex where a colony is formed. Observe $M_d$ is infinite if and only if the process survives. 

To study properties of $M_d$ we need the definitions of a few technical quantities.   

\begin{defn}\label{parametros}
Consider $\procR$ We define the quantities
\[\begin{array}{cll}
     \alpha &:=& d \mathbb{E} \left[\frac{N}{N+d}\right] \\ \\
     \beta  &:=& (d+1) \mathbb{E} \left [\frac{N}{N+d}\right] = \alpha + \mathbb{E} \left [\frac{N}{N+d}\right]\\  \\
     D &:=& \max \left \{ 2; \frac{\beta}{\beta - \mathbb{P}(N \neq 0)} \right\} \\ \\
B& :=& d(d-1)\left[\mathbb{E} \left(\frac{N(N-1)}{(N+d-1)(N+d-2)} \right) \right] 
\end{array}\]
\end{defn}

\begin{lem}\label{lemparticular} For $\procpR$, the quantities of the Definition~ \ref{parametros} are given by
\[\begin{array}{cll}
\alpha &=& \frac{dp(\lambda +1)}{(\lambda p +1)^2} \left [ \lambda p + 1 -d\Phi \left( \frac{\lambda p }{\lambda p + 1}, 1,d+1\right ) \right ] \\ \\
\beta &=& \frac{(d+1)p(\lambda +1)}{(\lambda p +1)^2} \left [ \lambda p + 1 -d\Phi \left( \frac{\lambda p }{\lambda p + 1}, 1,d+1\right ) \right ] \\ \\
D &=& \max \left \{ 2; \frac{\beta(\lambda p+1)}{\beta + p(\beta \lambda -\lambda -1)} \right \} \\ \\
B &=& \frac{d\lambda p^2(d-1)(\lambda +1)}{(\lambda p + 1)^3} \left [ \frac{ d^2 +d(\lambda p -2) +2}{d} - \frac{(d-1)(d + 2\lambda p)}{(\lambda p + 1)} \Phi \left( \frac{\lambda p }{\lambda p + 1}, 1,d+1\right ) \right ] 
\end{array}\]
  \end{lem}
 
\begin{teo}
\label{T: alcance}
Consider $M_d$, the reach of $\procR$.  Assume that \begin{displaymath}
\mathbb{E} \left [\frac{N}{N+d}\right]  < \frac{1}{d+1}.
\end{displaymath}
We have that 

\begin{displaymath}
\frac{[1+D(1-\beta)][1-\beta^{n+1}]}{1+ D(1-\beta)-\beta^{n+1}} \leq \mathbb{P}(M_d \leq n) \leq \frac{[1 + \frac{\alpha(1-\alpha)}{B}](1-\alpha^{n+1})}{ 1 + \frac{\alpha(1-\alpha)}{B} - \alpha^{n+1} }
\end{displaymath}

where $\alpha, \beta, B$ and $D$ are given in Definition~ \ref{parametros}. 
\end{teo}

\begin{teo}\label{T1: LEITL}
Let $M_d$ be the reach of  $\procR$. When $d \to \infty$ we have that
\begin{displaymath}
M_d \overset{\mathcal{D}}{\to} M,
\end{displaymath}
where $\mathbb{P}( M \leq m) = g_{m+1}(0)$, $g(s) = \mathbb{E}(s^N) $ and $g_{m+1}(s) = \overset{ m+1 \textrm { times }} {g(g(\cdots g(s)) \cdots )}$.
\end{teo}

\begin{exa}
Let $M_d$ be the reach of $\procpR$. We have that $M_d \overset{\mathcal{D}}{\to} M$ where
\[
\mathbb{P}(M \leq n) = \frac{1- [(\lambda +1)p]^{n+1}}{1 - \frac{\lambda p}{1-p}[(\lambda +1)p]^{n+1}} \textrm{ and } \mathbb{E}(M) = \frac{(1-p-\lambda p)}{\lambda p} \sum_{n=0}^{\infty}\frac{[(\lambda +1)p]^{n+1}}{\frac{1-p}{\lambda p} - [(\lambda +1)p]^{n+1}}.
\]
\end{exa}

\subsection{Number of colonies}

\begin{teo}
\label{T: MCC2HN}
Consider $\proc$ and let $I_d$ be the number of colonies created during the process. If
\begin{displaymath}
\mathbb{E} \left [ \frac{N}{N+d} \right]  < \frac{1}{d+1}
\end{displaymath}
then
\begin{displaymath}
\frac{1+ \mathbb{E} \left [ \frac{N}{N+d} \right]  }{1- d\mathbb{E} \left [ \frac{N}{N+d} \right] } \leq \mathbb{E}(I_d) \leq \frac{1}{1- (d+1)\mathbb{E} \left [ \frac{N}{N+d} \right] } 
\end{displaymath}
In addition, if $\mathbb{E}(N) < 1$ (the subcritical case),
\[
\lim_{d \to \infty} \mathbb{E}(I_d) = \frac{1}{1 - \mathbb{E}(N)}. \]
\end{teo}

\begin{cor}
\label{C: MCC1HPSEQ}
Consider $\procp$ and let $I_d$ be the number of colonies created during the process. If 
\[
\Phi \left( \frac{\lambda p }{\lambda p + 1}, 1,d+1\right ) > \frac{(\lambda p +1)[p(\lambda d + d + 1)-1]}{pd(d+1)(\lambda +1)},
\]
then
\begin{displaymath}
\frac{d+\alpha}{d(1-\alpha)} \leq \mathbb{E}(I_d) \leq \frac{d}{d-(d+1)\alpha},   
\end{displaymath}
where $\alpha$ is given in Lemma~\ref{lemparticular}. Besides, if $(\lambda +1)p<1$ then
\begin{align*}
\lim_{d \to \infty}\mathbb{E}(I_d) = \frac{1}{1 -(\lambda +1)p}.
\end{align*}
\end{cor}

\begin{exa}
Let $I_d$ be the total number of colonies created in $\procpR$ with $\lambda=9$ and $p=\frac{99}{1000}$. If $d=800$ we have that \[ 74.5761 \leq \mathbb{E}(I_{800}) \leq 82.0181. \] Besides, 
$ \displaystyle \lim_{d \to \infty}\mathbb{E}(I_d) = 100.$
In the model with scheme independent dispersion, \cite[Theorem 3.23]{MRV2018}, we have \[ 81.1729 \leq \mathbb{E}(I_{800}) \leq 90.0900. \] As in example~\ref{E: Id}, due to the intersection of the ranges obtained from the rigorous results currently
available, it is not possible to make a definitive comparison.
\end{exa}

\subsection{Extinction time}

\begin{defn}
Let be $\eta_t$ the process $\text{C}(\mathbb{G};\mathcal{N},\mathcal{L})$. We define the extinction time of the process $\text{C}(\mathbb{G};\mathcal{N},\mathcal{L})$ by 
$$\tau:=\inf\{t>0: |\eta_t|=0\}.$$
\end{defn}

\begin{teo}
\label{T: exttime}
Consider $\procR$ and let $\tau_u(d)$ be the extinction time of the process.  Assume that
\begin{displaymath} 
\mathbb{E} \left [\frac{N}{N+d}\right]  < \frac{1}{d+1}. 
\end{displaymath}
Then
\[
\int_{0}^{1} \frac{1-s}{G_L(s) - s}ds \leq \mathbb{E}(\tau_u(d)) \leq \int_{0}^{1} \frac{1-s}{G_U(s) - s}ds
\]
where
\[
G_L(s) = \sum_{y=0}^{d}s^y \left[ \sum_{n=y}^{\infty} \left[
\mathbb{P}(N=n) \frac{\dbinom{d}{y}\dbinom{n}{y}}{\dbinom{n+d}{d}}
\right] \right]
\]
and
\[
G_U(s) = \mathbb{P}(N=0) + \sum_{y=1}^{d+1} \left[ \dbinom{d+1}
{y}s^y \sum_{n=y}^{\infty} \left[ \mathbb{P}(N=n) \frac{\dbinom{n-1}
{y-1}}{\dbinom{n+d}{d}} \right] \right].
\]

\end{teo}

\section{Proofs}
\label{S: Proofs}

To prove our results we define auxiliary processes on the graphs $\mathbb{T}^d$ and $\mathbb{T}^d_+,$ whose understanding will provide
bounds for the processes defined in Section 2. 

In the first two auxiliary processes, denoted by $\up$ and $\upP$,
every time a catastrophe occurs in a colony,  according to uniform dispersion, the survival individuals disperse over neighboring vertices that are further from the origin than the vertex where that colony was placed. In other words, individuals do not disperse to sites that have already been colonized. We refer to this process as \textit{Self Avoiding}.

In the last two auxiliary processes, represented by $\low$ and $\lowP$, surviving individuals disperse to any of the $d+1$ neighboring vertices according to the uniform dispersion. However, those who move backward towards the origin die, as this direction is deemed inhospitable or infertile. We refer to this process as \textit{Move Forward or Die}.
In both processes, every new colony starts with only one individual.

The main idea behind the proofs
is the identification of an underlying branching process related to the models. After doing that we can apply results of the theory of Branching
Processes, including the less known and new results presented in~\cite{MRV2018}.

\subsection{The Self Avoiding Model}\label{S:Uunif}
\hfill

\begin{prop}
\label{P: CCSRSR2}
Consider $\up$ and let $V_d$ be the survival event. We have that $\mathbb{P}(V_d) > 0 $ if and only if
\begin{displaymath}
\mathbb{E} \left [ \frac{N}{N+d-1} \right]  > \frac{1}{d}.
\end{displaymath}
\end{prop}
\textit{Proof of Proposition~\ref{P: CCSRSR2}}\\
Next, observe that the
process $\upP$ behaves as a homogeneous branching process. Every vertex $x\in\mathbb{T}^d$ which is colonized
produces $Y$ new colonies (whose distribution depends only on $\mathcal{N}$ and $\mathcal{L}$) on the $d$ neighbor vertices which are located further from the origin than $x$ is. By numbering those $d$ vertices, we can represent the random variable $Y$ as 
\begin{displaymath}
 Y = \sum_{i=1}^{d}\mathcal{I}_i
\end{displaymath}
where $\mathcal{I}_i$ is the indicator variable of event the vertex $i$ receives at least one individual.

Thus,
\begin{displaymath}
 \mathbb{E}(Y) = \sum_{i=1}^{d}\mathbb{P}(\mathcal{I}_i=1) = d \sum_{n\geq1} \mathbb{P}(\mathcal{I}_1=1|N=n)\mathbb{P}(N = n) .
\end{displaymath}
Now, note that
\begin{displaymath}
\mathbb{P}(\mathcal{I}_1 = 1 | N = n) = \frac{n}{n+d-1},
\end{displaymath}
and so
\begin{equation}\label{eq:mediadesc}
  \mathbb{E}(Y) =   d \sum_{n\geq1} \left [ \frac{n}{n+d-1} \mathbb{P}(N = n) \right]  = d \mathbb{E} \left ( \frac{N}{N+d-1} \right ).   
\end{equation}
From the theory of homogeneous branching processes we see that $\upP$ (and also $\up$) survives if and only if $ \mathbb{E}(Y)>1.$

\begin{prop} \label{P: CCSRSR2V}
Consider the process $\up$. Let $V_d$ be the survival event and $I_d$ the number of colonies created during the process. Then
\begin{displaymath}
\mathbb{P}(V_d) = \sum_{r=1}^{d+1}\left [(1 - \psi^r)\binom{d+1}
{r}\sum_{n=r}^{\infty}\frac{\binom{n-1}{r-1}}{\binom{n+d}{d}}\mathbb
{P}(N=n) \right]
\end{displaymath}
where $\psi$, the extiuncion probability for the process $\upP$, is the smallest non-negative solution of 
\begin{displaymath}
\sum_{y=1}^{d} \left[ s^y\binom{d}{y}\sum_{n=y}^{\infty}\frac{\binom
{n-1}{y-1}}{\binom{n+d-1}{d-1}}\mathbb{P}(N=n) \right] = s - \mathbb{P}
(N=0)
\end{displaymath}
On the subcritical regime, which means
\begin{displaymath}
\mathbb{E} \left [ \frac{N}{N+d-1} \right]  < \frac{1}{d}
\end{displaymath}
it holds that
\[
\mathbb{E}(I_d) = 1 + \frac{(d+1)\mathbb{E} \left ( \frac{N}{N+d} \right )}{1 - d\mathbb{E} \left (\frac{N}{N+d-1} \right )}.
\]

\end{prop}
\textit{Proof of Proposition~\ref{P: CCSRSR2V}}\\
When a colony placed at the origin collapses, $Y_R$ new colonies are formed by the survivors at its neighboring vertices. If a colony located outside the origin collapses, $Y$ new colonies are created on neighboring vertices.
Observe that the
process $\up$ behaves as a non-homogeneous branching process $\{Z_n\}_{n \geq 0}$ with
\[
Z_o = 1, Z_{n+l} = \sum_{i=1}^{Z_n}X_{n,i}, n \geq 0,
\]
where $X_{1,1}$ and $X_{n,i}, n >1$ are distributed as $Y_R$ and $Y$, respectively. So, we have that 
\begin{displaymath}
\mathbb{P}(V_d)  = \sum_{r=0}^{d+1} \mathbb{P}(V_d| Y_R = r) \mathbb{P}
(Y_R = r).
\end{displaymath}
Given that $Y_R = r$ one have $r$ independent $\upP$ processes living on $r$ independent rooted trees. So, we have that $\mathbb{P}(V_d^C | Y_R = r) = \psi^r, r =0,1,2,
\cdots, d+1$, where $\psi$ is the smallest non-negative solution of ${E}(s^Y)=s.$

By using the Total Probability Theorem,
\begin{align}\label{distY}
\mathbb{P}(Y_R=y) &=\sum_{n=r}^{\infty} \left[ \mathbb{P}(N=n) \frac
{\dbinom{d+1}{y}\dbinom{n-1}{y-1}}{\dbinom{n+d}{d}}\right] \textrm {
for } y =0,1,2, \cdots, d+1.\nonumber\\
\mathbb{P}(Y=y) &=\sum_{n=y}^{\infty} \left[ \mathbb{P}(N=n) \frac
{\dbinom{d}{y}\dbinom{n-1}{y-1}}{\dbinom{n+d-1}{d-1}}\right] \textrm {
for } y =0,1,2, \cdots, d.
\end{align}
The probability generating function of $Y$ is 
\begin{equation}\label{fgpY}
G_U(s)=\mathbb{E}(s^Y)  = \mathbb{P}(N=0) + \sum_{y=1}^{d} \left[ \dbinom{d}
{y}s^y \sum_{n=y}^{\infty} \left[ \mathbb{P}(N=n) \frac{\dbinom{n-1}
{y-1}}{\dbinom{n+d-1}{d-1}} \right] \right]
\end{equation}

Thus,
\begin{displaymath}
\mathbb{P}(V_d) = \sum_{r=1}^{d+1}\left [(1 - \psi^r)\binom{d+1}
{r}\sum_{n=r}^{\infty}\frac{\binom{n-1}{r-1}}{\binom{n+d}{d}}\mathbb
{P}(N=n) \right].
\end{displaymath}
As for the second part of the proposition, note that
\begin{displaymath}
\mathbb{E}(I_d)  = \sum_{r=0}^{d+1} \mathbb{E}(I_d| Y_R = r) \mathbb{P}
(Y_R = r)
\end{displaymath}
Besides, from (\ref{eq:mediadesc}) we have 
\begin{displaymath}
\mathbb{E}(I_d| Y_R = r) = r\mu + 1 \textrm{ where } \mu = \left [ 1 - d
\mathbb{E} \left ( \frac{N}{N+d-1} \right)   \right ]^{-1},
\end{displaymath}
see Stirzaker~\cite[Exercise 2b, p. 280]{Stirzaker}.

\begin{prop}
\label{P: UlimPS}
Consider the process $\up$.  Let $V_d$ be the survival event and $I_d$ the number of colonies created during the process. Then
\begin{equation}
\label{limPVD}
\lim_{d \to \infty} \mathbb{P}(V_d) = 1 - \nu
\end{equation}
where  $\nu$ is the smallest non-negative solution of $ \mathbb{E}(s^N) = s$. Besides that, if 
$\mathbb{E}(N) < 1$ (the subcritical case) then
\begin{equation}
\label{limEIDU}
\lim_{d \to \infty} \mathbb{E}(I_d) = \frac{1}{1 - \mathbb{E}(N)}.
\end{equation}
\end{prop}
\begin{proof}[Proof of Proposition~\ref{P: UlimPS}]
In order to prove~(\ref{limPVD}) one has to apply \cite[Proposition 4.2]{MRV2018}, observing that $Y \overset{\mathcal D}{\to} N$ and $Y_{R} \overset{\mathcal D}{\to} N$, when $d\to\infty.$
Moreover to prove~(\ref{limEIDU}) observe that
\[
\lim_{d \to \infty} \mathbb{E}(I_d) = \lim_{d \to \infty} \sum_{r=0}^{d+1} \mathbb{E}(I_d|Y_{R}=r)\bbP(Y_{R}=r).\]
As  $Y \overset{\mathcal D}{\to} N$ and $Y_{R} \overset{\mathcal D}{\to} N $ when $d\to\infty$ then
\[ \lim_{d \to \infty}  \mathbb{E}(I_d|Y_{R}=r) = \lim_{d \to \infty} r\frac{1}{1-\mathbb{E}(Y)} +1 = \frac{r}{1-\mathbb{E}(N)}+1 \]
and the result follows from the Dominated Convergence Theorem~\cite[Theorem 9.1 p. 26]{Thorisson}.
\end{proof}

\begin{prop} \label{P: LIMTSRSR2}
Let $M_d$ be the reach of the process $\upP$, that is, the distance from the origin to the farthest vertex where a colony is formed. 
Assuming
\begin{displaymath}
\mathbb{E} \left [ \frac{N}{N+d-1} \right]  < \frac{1}{d}
\end{displaymath}
Then
\[
\frac{\left [1+D(1- \mu_u \right](1- B_u^{n+1})}{1+D(1-\mu_u - \mu_u^{n+1})} \leq \mathbf{P}(M_d \leq n) \leq \frac{\left[1+ \frac{\mu_u(1- \mu_u)}{B_u} \right](1- \mu_u^{n+1})}{1+ \frac{\mu_u(1- \mu_u)}{B_u} - \mu_u^{n+1}}
\]

where
\begin{align*}
\mu_u &=  d \mathbb{E} \left [\frac{N}{N+d-1}\right] \\
D &= \max \left \{ 2; \frac{\mu_u}{\mu_u - \mathbb{P}(N \neq 0)} \right\} \\
B_u& =d(d-1)\left[\mathbb{E} \left(\frac{N(N-1)}{(N+d-1)(N+d-2)} \right) \right]
\end{align*}
Moreover,
\begin{displaymath}
M_d \overset{\mathcal D}{\to} M,
\end{displaymath}
where $\mathbb{P}( M \leq m) = g_{m+1}(0)$, being $g(s) = \mathbb{E}(s^N) $ and $g_{m+1}(s) = \overset{ m+1 \textrm { times }} {g(g(\cdots g(s)) \cdots )}$.
\end{prop}

\textit{Proof of Proposition~\ref{P: LIMTSRSR2}}\\
Every vertex $x\in\mathbb{T}^d$ which is colonized
produces $Y$ new colonies (whose distribution depends only on $\mathcal{N}$ and $\mathcal{L}$) on the $d$ neighbor vertices which located are further from the origin than $x$ is. By numbering those $d$ vertices, we can represent the random variable $Y$ as 
\begin{displaymath}
 Y = \sum_{i=1}^{d}\mathcal{I}_i
\end{displaymath}
where $\mathcal{I}_i$ is the indicator variable of event the vertex $i$ receives at least one individual.

From (\ref{eq:mediadesc}) we obtained $\mathbb{E}(Y).$ Now

\begin{displaymath}
 Y^2 = \left (\sum_{i=1}^{d}\mathcal{I}_i \right )^2 = \sum_{i=1}^{d}\mathcal{I}_i^2 + 2\sum_{1 \leq i < j \leq d}\mathcal{I}_i\mathcal{I}_j
\end{displaymath}
\begin{displaymath}
 \mathbb{E} \left(Y^2 \right ) =  d \mathbb{E} \left (\mathcal{I}_1^2 \right ) + d(d-1) \mathbb{E}(\mathcal{I}_1\mathcal{I}_2)
\end{displaymath}
On the other hand
\begin{align*}
 \mathbb{E} \left (\mathcal{I}_1^2 \right ) & = \mathbb{P}(\mathcal{I}_1 = 1) = \sum_n \left [\mathbb{P}(\mathcal{I}_1=1|N=n)\mathbb{P}(N = n) \right] = \sum_{n} \left (\frac{n}{n+d-1} \right ) \mathbb{P}(N = n) \\ & = \mathbb{E} \left (\frac{N}{N+d-1} \right )
\end{align*}
and
\begin{align*}
 \mathbb{E} \left (\mathcal{I}_1 \mathcal{I}_2 \right ) & = \mathbb{P}(\mathcal{I}_1 = 1; \mathcal{I}_2 = 1) = \sum_n \left [\mathbb{P}(\mathcal{I}_1=1; \mathcal{I}_2 = 1|N=n)\mathbb{P}(N = n) \right]
\end{align*}
where
\begin{align*}
\mathbb{P}(\mathcal{I}_1=1; \mathcal{I}_2 = 1|N=n) = \frac{n(n-1)}{(n+d-1)(n+d-2)}
\end{align*}
Then
\begin{displaymath}
 \mathbb{E} \left(Y^2 \right ) = d\mathbb{E} \left (\frac{N}{N+d-1} \right ) +d(d-1)\mathbb{E} \left (\frac{N(N-1)}{(N+d-1)(N+d-2)} \right )
\end{displaymath}
\begin{align*}
\mathbb{E} \left(Y(Y-1) \right ) = d(d-1)\mathbb{E} \left (\frac{N(N-1)}{(N+d-1)(N+d-2)} \right )
\end{align*}
Then the result follows from  ~\cite[Theorem 1 p. 331]{AA} with $\mu_u=\mathbb{E}(Y)$ and $B_u=\mathbb{E} \left(Y(Y-1) \right ).$

The convergence $M_d \overset{\mathcal D}{\to} M$ follows from the fact that $Y \overset{\mathcal D}{\to} N$ when $ d \to \infty$ and from \cite[Proposition 4.1]{MRV2018}.

\begin{prop}\label{ExtTimeU}
   Let $\tau_u(d)$ be the extinction time of the process $\upP$. If 
   \begin{displaymath}
\mathbb{E} \left [ \frac{N}{N+d-1} \right]  < \frac{1}{d}
\end{displaymath} 
then 
\[\mathbb{E}[\tau_u(d)]=\displaystyle\int_0^1\frac{1-y}{G_U(y)-y}dy.\] where $G_U(s)$ is given by (\ref{fgpY}).
\end{prop}
\begin{proof}
Let $Z_t$ be the number of colonies at time $t$ in the model $\upP$. Observe that $Z_t$ is a continuous-time branching process with  $Z_0=1$. Each particle (colony) in $Z_t$ survives an exponential time of rate 1 and right before death produces $Y\leq d$ particles (colonies are created right after a catastrophe) with distribution given by (\ref{distY}) and probability generating function given by (\ref{fgpY}). 
If $G'_U(1)\leq1$, then $\mathbb{P}[\tau_u(d)<\infty]=1$ and the result following from  Narayan~\cite{PN1982}.

\end{proof}

\subsection{Move Forward or Die Model}
\hfill

\begin{prop}
\label{P: CCSRCR2}
Consider $\low$ and let $V_d$ be the survival event. We have that $\mathbb{P}(V_d) > 0 $ if and only if
\begin{displaymath}
\mathbb{E} \left [ \frac{N}{N+d} \right]  > \frac{1}{d}.
\end{displaymath}
\end{prop}

\textit{Proof of Proposition~\ref{P: CCSRCR2}}\\

Firstly, note that for fixed $\mathcal{N}$ and $\mathcal{L}$, both processes $\low$ and $\lowP$ either both have positive probability of survival or neither do

Next, observe that the
process $\lowP$ behaves as a homogeneous branching process. Every vertex $x\in\mathbb{T}^d$ which is colonized
produces $Y$ new colonies (whose distribution depends only on $\mathcal{N}$ and $\mathcal{L}$) on the $d$ neighbor vertices which are located further from the origin than $x$ is. By numbering those $d$ vertices, we can represent the random variable $Y$ as 
\begin{displaymath}
 Y = \sum_{i=1}^{d}\mathcal{I}_i
\end{displaymath}
where $\mathcal{I}_i$ is the indicator variable of event the vertex $i$ receives at least one individual.
Thus,
\begin{displaymath}
 \mathbb{E}(Y) = \sum_{i=1}^{d}\mathbb{P}(\mathcal{I}_i=1) = d \sum_{n\geq1} \mathbb{P}(\mathcal{I}_1=1|N=n)\mathbb{P}(N = n) .
\end{displaymath}
Now, note that
\begin{displaymath}
\mathbb{P}(\mathcal{I}_1 = 1 | N = n) = \frac{n}{n+d},
\end{displaymath}
and so
\begin{equation}\label{eq:mediadescL}
  \mathbb{E}(Y) =   d \sum_{n\geq1} \left [ \frac{n}{n+d} \mathbb{P}(N = n) \right]  = d \mathbb{E} \left ( \frac{N}{N+d} \right ).   
\end{equation}
From the theory of homogeneous branching processes we see that $\lowP$ (and also $\low$) survives if and only if $ \mathbb{E}(Y)>1.$

\begin{prop} \label{P: CCSRCR2V}
Consider the process $\low$.  Let $V_d$ be the survival event and $I_d$ the number of colonies created during the process. Then 
\begin{displaymath}
\mathbb{P}(V_d) = \sum_{r=1}^{d+1}\left [(1 - \rho^r)\binom{d+1}
{r}\sum_{n=r}^{\infty}\frac{\binom{n-1}{r-1}}{\binom{n+d}{d}}\mathbb
{P}(N=n) \right]
\end{displaymath}
where $\rho$, the extinction probability for the process $\lowP$, is the smallest non-negative solution of
\begin{displaymath}
\sum_{y=0}^{d} \left[ s^y\binom{d}{y}\sum_{n=y}^{\infty}\frac{\binom
{n}{y}}{\binom{n+d}{d}}\mathbb{P}(N=n) \right] = s.
\end{displaymath}
On the subcritical regime, which means
\begin{displaymath}
\mathbb{E} \left [ \frac{N}{N+d} \right]  < \frac{1}{d}
\end{displaymath}
it holds that

\[
\mathbb{E}(I_d)  = \frac{1 + \mathbb{E}\left ( \frac{N}{N+d} \right) }{1 - d\mathbb{E}\left ( \frac{N}{N+d} \right)}.
\]
\end{prop}

\textit{Proof of Proposition~\ref{P: CCSRCR2V}}\\
When a colony placed at the origin collapses, $Y_R$ new colonies are formed by the survivors at its neighboring vertices. If a colony located outside the origin collapses, $Y$ new colonies are created on neighboring vertices.
Observe that the
process $\low$ behaves as a non-homogeneous branching process $\{Z_n\}_{n \geq 0}$ with
\[
Z_0 = 1, Z_{n+l} = \sum_{i=1}^{Z_n}X_{n,i}, n \geq 0,
\]
where $X_{1,1}$ and $X_{n,i}, n >1$ are distributed as $Y_R$ and $Y,$ respectively. So, we have that 
\begin{displaymath}
\mathbb{P}(V_d)  = \sum_{r=0}^{d+1} \mathbb{P}(V_d| Y_R = r) \mathbb{P}
(Y_R = r).
\end{displaymath}
Given that $Y_R = r$ one have $r$ independent $\lowP$ processes living on $r$ independent rooted trees. So, we have that $\mathbb{P}(V_d^C | Y_R = r) = \rho^r, r =0,1,2,
\cdots, d+1$, where $\rho$ is the smallest non-negative solution of ${E}(s^Y)=s.$

By using the Total Probability Theorem,

\begin{align}
\mathbb{P}(Y_R=y) &=\sum_{n=y}^{\infty} \left[ \mathbb{P}(N=n) \frac
{\dbinom{d+1}{y}\dbinom{n-1}{y-1}}{\dbinom{n+d}{d}}\right] \textrm {
for } y =0,1,2, \cdots, d+1.\nonumber\\
\mathbb{P}(Y=y) &= \sum_{n=y}^{\infty} \left[ \mathbb{P}(N=n) \frac
{\dbinom{d}{y}\dbinom{n}{y}}{\dbinom{n+d}{d}}\right] \textrm { for }
y =0,1,2, \cdots, d.\label{distYL}
\end{align}

The probability generating function of $Y$ is 
\begin{equation}\label{fgpYL}
G_L(s)=\mathbb{E}(s^Y)  = \sum_{y=0}^{d}s^y \left[ \sum_{n=y}^{\infty} \left[
\mathbb{P}(N=n) \frac{\dbinom{d}{y}\dbinom{n}{y}}{\dbinom{n+d}{d}}
\right] \right].
\end{equation}

Thus,

\begin{displaymath}
\mathbb{P}(V_d) = \sum_{r=1}^{d+1}\left [(1 - \rho^r)\binom{d+1}
{r}\sum_{n=r}^{\infty}\frac{\binom{n-1}{r-1}}{\binom{n+d}{d}}\mathbb
{P}(N=n) \right]
\end{displaymath}
As for the second part of the proposition, note that
\begin{displaymath}
\mathbb{E}(I_d)  = \sum_{r=0}^{d+1} \mathbb{E}(I_d| Y_R = r) \mathbb{P}
(Y_R = r)
\end{displaymath}
Besides, from (\ref{eq:mediadescL}) we have
\begin{displaymath}
\mathbb{E}(I| Y_R = r) = r\mu + 1 \textrm{ onde } \mu = \left [ 1 - d
\mathbb{E} \left ( \frac{N}{N+d} \right)   \right ]^{-1},
\end{displaymath}
see Stirzaker~\cite[Exercise 2b, p. 280]{Stirzaker}.

\begin{prop}
\label{P: LlimPS}
Consider the process $\low$. Let $V_d$ be the survival event and $I_d$ the number of colonies created during the process. Then
\begin{equation}
\label{LlimPVD}
\lim_{d \to \infty} \mathbb{P}(V_d) = 1 - \nu
\end{equation}
where  $\nu$ is the smallest non-negative solution of $ \mathbb{E}(s^N) = s$. Besides that, if 
$\mathbb{E}(N) < 1$ (the subcritical case) then
\begin{equation}
\label{limEIDL}
\lim_{d \to \infty} \mathbb{E}(I_d) = \frac{1}{1 - \mathbb{E}(N)}.
\end{equation}
\end{prop}
\begin{proof}[Proof of Proposition~\ref{P: LlimPS}]
In order to prove~(\ref{LlimPVD}) one has to apply \cite[Proposition 4.2]{MRV2018}, observing that $Y \overset{\mathcal D}{\to} N$ and $Y_{R} \overset{\mathcal D}{\to} N$, when $d\to\infty.$
Moreover to prove~(\ref{limEIDL}) observe that
\[
\lim_{d \to \infty} \mathbb{E}(I_d) = \lim_{d \to \infty} \sum_{r=0}^{d+1} \mathbb{E}(I_d|Y_{R}=r)\bbP(Y_{R}=r).\]
As  $Y \overset{\mathcal D}{\to} N$ and $Y_{R} \overset{\mathcal D}{\to} N $ when $d\to\infty$ then
\[ \lim_{d \to \infty}  \mathbb{E}(I_d|Y_{R}=r) = \lim_{d \to \infty} r\frac{1}{1-\mathbb{E}(Y)} +1 = \frac{r}{1-\mathbb{E}(N)}+1 \]
and the result follows from the Dominated Convergence Theorem~\cite[Theorem 9.1 p. 26]{Thorisson}.
\end{proof}

\begin{prop} \label{P: LIMTSRCR2}
Let $M_d$ be the reach of the process $\lowP$, that is, the distance from the origin to the farthest vertex where a colony is formed. 
Assuming
\begin{displaymath}
\mathbb{E} \left [ \frac{N}{N+d} \right]  < \frac{1}{d}
\end{displaymath}
Then
\[
\frac{\left [1+D(1- \mu_l \right](1- B_l^{n+1})}{1+D(1-\mu_l - \mu_l^{n+1})} \leq \mathbf{P}(M_d \leq n) \leq \frac{\left[1+ \frac{\mu_l(1- \mu_l)}{B_l} \right](1- \mu_l^{n+1})}{1+ \frac{\mu_L(1- \mu_l)}{B_l} - \mu_l^{n+1}}
\]

where
\begin{align*}
\mu_l &=  d \mathbb{E} \left [\frac{N}{N+d}\right] \\
D &= \max \left \{ 2; \frac{\mu_l}{\mu_l - \mathbb{P}(N \neq 0)} \right\} \\
B_l& =d(d-1)\left[\mathbb{E} \left(\frac{N(N-1)}{(N+d-1)(N+d)} \right) \right]
\end{align*}
Moreover,
\begin{displaymath}
M_d \overset{\mathcal D}{\to} M,
\end{displaymath}
where $\mathbb{P}( M \leq m) = g_{m+1}(0)$, being $g(s) = \mathbb{E}(s^N) $ and $g_{m+1}(s) = \overset{ m+1 \textrm { times }} {g(g(\cdots g(s)) \cdots )}$.
\end{prop}
\textit{Proof of Proposition~\ref{P: LIMTSRCR2}}\\
Every vertex $x\in\mathbb{T}^d$ which is colonized
produces $Y$ new colonies (whose distribution depends only on $\mathcal{N}$ and $\mathcal{L}$) on the $d$ neighbor vertices which located are further from the origin than $x$ is. By numbering those $d$ vertices, we can represent the random variable $Y$ as 
\begin{displaymath}
 Y = \sum_{i=1}^{d}\mathcal{I}_i
\end{displaymath}
where $\mathcal{I}_i$ is the indicator variable of event the vertex $i$ receives at least one individual.

From (\ref{eq:mediadescL}) we obtained $\mathbb{E}(Y).$ Now

\begin{displaymath}
 Y^2 = \left (\sum_{i=1}^{d}\mathcal{I}_i \right )^2 = \sum_{i=1}^{d}\mathcal{I}_i^2 + 2\sum_{1 \leq i < j \leq d}\mathcal{I}_i\mathcal{I}_j
\end{displaymath}
\begin{displaymath}
 \mathbb{E} \left(Y^2 \right ) =  d \mathbb{E} \left (\mathcal{I}_1^2 \right ) + d(d-1) \mathbb{E}(\mathcal{I}_1\mathcal{I}_2)
\end{displaymath}
On the other hand
\begin{align*}
 \mathbb{E} \left (\mathcal{I}_1^2 \right ) & = \mathbb{P}(\mathcal{I}_1 = 1) = \sum_n \left [\mathbb{P}(\mathcal{I}_1=1|N=n)\mathbb{P}(N = n) \right] = \sum_{n} \left (\frac{n}{n+d} \right ) \mathbb{P}(N = n) \\ & = \mathbb{E} \left (\frac{N}{N+d} \right )
\end{align*}
and
\begin{align*}
 \mathbb{E} \left (\mathcal{I}_1 \mathcal{I}_2 \right ) & = \mathbb{P}(\mathcal{I}_1 = 1; \mathcal{I}_2 = 1) = \sum_n \left [\mathbb{P}(\mathcal{I}_1=1; \mathcal{I}_2 = 1|N=n)\mathbb{P}(N = n) \right]
\end{align*}
where
\begin{align*}
\mathbb{P}(\mathcal{I}_1=1; \mathcal{I}_2 = 1|N=n) = \frac{n(n-1)}{(n+d-1)(n+d)}
\end{align*}
 
Then
\begin{displaymath}
 \mathbb{E} \left(Y^2 \right ) = d\mathbb{E} \left (\frac{N}{N+d} \right ) +d(d-1)\mathbb{E} \left (\frac{N(N-1)}{(N+d-1)(N+d)} \right )
\end{displaymath}
\begin{align*}
\mathbb{E} \left(Y(Y-1) \right ) = d(d-1)\mathbb{E} \left (\frac{N(N-1)}{(N+d-1)(N+d)} \right ).
\end{align*}
Then the result follows from  ~\cite[Theorem 1 p. 331]{AA} with  $\mu_l\,=\mathbb{E}(Y)$ and  $B_l\,=\mathbb{E} \left(Y(Y-1) \right )$.

The convergence $M_d \overset{\mathcal D}{\to} M$ follows from the fact that $Y\overset{\mathcal D}{\to} N$ when $ d \to \infty$ and from \cite[Proposition 4.1]{MRV2018}.

\begin{prop}\label{ExtTimeL}
   Let $\tau_u(d)$ be the extinction time of the process $\lowP$. If 
\begin{displaymath}
\mathbb{E} \left [ \frac{N}{N+d} \right]  < \frac{1}{d}
\end{displaymath}
then 
\[\mathbb{E}[\tau_u(d)]=\displaystyle\int_0^1\frac{1-y}{G_L(y)-y}dy.\] where $G_L(s)$ is given by (\ref{fgpYL}).
\end{prop}
\begin{proof}
    Analogous to the proof of Proposition \ref{ExtTimeU}.
\end{proof}

\subsection{Proofs of the Main Results}
\hfill
First we define a coupling between the processes $\proc$ and $\lowP$ in such a way that the latter is stochastically dominated by the former. Every colony in $\lowP$ is associated to a colony in $\proc$. As a consequence, if the process $\proc$ dies out, the same happens to $\lowP$.\\
At every catastrophe time at a vertex $x$ in the model $\proc$, a non-empty group of individuals that tries to colonize the neighbor vertex to $x$ which is closer to the origin than $x$ will create there a new colony provided that that vertex is empty. In the model $\lowP$ the same
non-empty group of individuals that tries to colonize the same vertex, immediately dies.\\
Next we define a coupling between the processes $\proc$ and $\upPd$ in such a way that the former is stochastically dominated by the latter. Every colony in $\proc$ can be associated to a colony in $\upPd$. Thus, if the process $\upPd$ dies out, the same happens to $\proc$.\\
At every catastrophe time at a vertex $x$ in the model $\proc$, we associate the neighbor vertex to $x$ which is closer to the origin than $x$ to the extra vertex on the model $\upPd$. In the model $\proc$, a non-empty group of individuals that tries to colonize the neighbor vertex to $x$ which is
closer to the origin than $x$ will create there a new colony provided that that vertex is empty. In the model $\upPd$ the same non-empty group of individuals that tries to colonize the extra vertex, founds a new colony there.\\

\textit{Proof of Theorem~\ref{C: MCC2H}}
The result follows from the fact that the process $\proc$ stochastically
dominates the process $\lowP$ and by its turn, is stochastically dominated by the process $\upPd$, together with Propositions \ref{P: CCSRSR2} and \ref{P: CCSRCR2}.\\

\textit{Proof of Corollary~\ref{C: MCC2HP}}
The result follows from Theorem~\ref{C: MCC2H} and Equation (\ref{distparticular}).\\

\textit{Proof of Theorem~\ref{T: MCCHPE2}}
The result follows from the fact that the process $\proc$ stochastically
dominates the process $\lowP$ and by its turn, is stochastically dominated by the process $\upPd$, together with Propositions \ref{P: CCSRSR2V} and \ref{P: CCSRCR2V}.\\

\textit{Proof of Theorem~\ref{T: MCCHPE1L}}
The result follows from the fact that the process $\proc$ stochastically
dominates the process $\lowP$ and by its turn, is stochastically dominated by the process $\upPd$, together with Propositions \ref{P: UlimPS} and \ref{P: LlimPS}.\\

\textit{Proof of Corollary~\ref{C: MCC1HPS}}
The result follows from Theorem~\ref{T: MCCHPE1L} and Equation (\ref{fgpparticular}).\\

\textit{Proof of Theorem~\ref{T: alcance}}
The result follows from the fact that the process $\proc$ stochastically
dominates the process $\lowP$ and by its turn, is stochastically dominated by the process $\upPd$, together with Propositions \ref{P: LIMTSRSR2} and \ref{P: LIMTSRCR2}.\\

\textit{Proof of Theorem~\ref{T1: LEITL}}
The result follows from the fact that the process $\proc$ stochastically
dominates the process $\lowP$ and by its turn, is stochastically dominated by the process $\upPd$, together with Propositions \ref{P: LIMTSRSR2} and \ref{P: LIMTSRCR2}.\\

\textit{Proof of Theorem~\ref{T: MCC2HN}}
The result follows from the fact that the process $\proc$ stochastically
dominates the process $\lowP$ and by its turn, is stochastically dominated by the process $\upPd$, together with Propositions \ref{P: CCSRSR2V} and \ref{P: CCSRCR2V}.\\

\textit{Proof of Theorem~\ref{T: exttime}} 
The result follows from the fact that the process $\proc$ stochastically
dominates the process $\lowP$ and by its turn, is stochastically dominated by the process $\upPd$, together with Propositions \ref{ExtTimeU} and \ref{ExtTimeL}.
\\

\section{Comparison with model of independent dispersion}
\label{S: compara}
In this section, we compare the Self Avoiding Models with independent dispersion and uniform dispersion. The model with uniform dispersion is $\up$ considered in Section \ref{S:Uunif} and  the model with independent dispersion is considered in \cite[Section 4.1]{MRV2018} and we denote it here by $\upi$. The following results indicate that independent dispersion is a better strategy than uniform dispersion in order to extend population survival. 

\begin{teo}\label{compara} Consider the processes $\upP$ and $\upPi$ with survives probability $P(V_u)$ and $P(V_i)$, respectively.  If $P(V_i)=0$ then $P(V_u)=0.$
\end{teo}

\begin{proof}
Observe that for a fixed $\mathbb{T}_+^d$, $\mathcal{N}$ and $\mathcal{L}$, both processes, $\upP$ and $\upPi$, behave as a homogeneous branching processes. Every vertex $x\in\mathbb{T}_+^d$ which is colonized
produces a random number of new colonies (whose distribution depends only on $\mathcal{N}$, $\mathcal{L}$ and of  dispersion scheme) on the $d$ neighbor vertices which located are further from the origin than $x$ is. We denote by $Y_u$ those number in the process $\upP$ and $Y_i$ in the process $\upPi$.  By numbering those $d$ vertices, we can represent the random variables $Y_u$ and $Y_i$ as 
\begin{displaymath}
 Y_u = \sum_{i=1}^{d}\mathcal{I}_i,
 \hspace{0.5cm}
 \text{ and }
 \hspace{0.5cm}
 Y_i = \sum_{i=1}^{d}\mathcal{U}_i,
\end{displaymath}
where $\mathcal{I}_i$ and $\mathcal{U}_i$ are the indicator variables of event the vertex $i$ receives at least one individual in the processes $\upP$ and $\upPi$, respectively.

Observe that
\begin{align*}
\mathbb{E}(Y_u | N = n) &= \sum_{i=1}^{d}\mathbb{E}(\mathcal{U}_i | N = n)  = \sum_{i=1}^{d}\mathbb{P}(\mathcal{U}_i =1 | N = n)  = d\mathbb{P}(\mathcal{U}_1 =1 | N = n)  \\ & = \frac{dn}{n+d-1}.
\end{align*}

\begin{align*}
\mathbb{E}(Y_i | N = n) &= \sum_{i=1}^{d}\mathbb{E}(\mathcal{I}_i | N = n)  = \sum_{i=1}^{d}\mathbb{P}(\mathcal{I}_i =1 | N = n)  = d\mathbb{P}(\mathcal{I}_1 =1 | N = n)  \\ & = d \left ( 1 - \left (\frac{d-1}{d} \right )^n \right ).
\end{align*}

Note that for all $n\geq0$ and $d\geq1 $ integers, 
\[
1- \left (\frac{d-1}{d} \right )^n \geq \frac{n}{n+d-1}.
\]

Thus,
\begin{equation}\label{inequality}
    \mathbb{E}(Y_i ) = \mathbb{E} [\mathbb{E}(Y_i | N )] \geq 
 \mathbb{E} [\mathbb{E}(Y_u | N )] = \mathbb{E}(Y_u ).
\end{equation}

The result is obtained through the comparison of the mean numbers of offspring in the branching processes.
\end{proof}

\begin{teo}\label{th:coloniascriadas}  Let $I_u(d)$ and $I_i(d)$  be the numbers of colonies created during the processes $\up$ and $\upi$, respectively. 
\begin{enumerate}
    \item Suppose that 
\[
\mathbb{E} \left [ \frac{N}{N+d-1} \right]  < \frac{1}{d}.
\]
Then
\[
\mathbb{E}(I_u(d)) =  1 + \frac{(d+1)\mathbb{E} \left ( \frac{N}{N+d} \right )}{1 - d\mathbb{E} \left (\frac{N}{N+d-1} \right )}.
\]
\item Let $G(s)$ be the probability generating function of $N$. Suppose that 
\[
G \left ( \frac{d-1}{d} \right ) < \frac{d-1}{d}
\]
then
\[
\mathbb{E}(I_i(d)) = 1 + \frac{(d+1) \left ( 1- G \left ( \frac{d}{d+1}\right ) \right )}{1 - d \left ( 1- G \left ( \frac{d-1}{d}\right )  \right )} .
\]

\item For every $d\geq2,$ it holds that $\mathbb{E}(I_u(d))\leq \mathbb{E}(I_i(d))$. Besides, if $\mathbb{E}(N)<1,$
\[
\lim_{d \to \infty} \mathbb{E}(I_i(d)) = \lim_{d \to \infty} \mathbb{E}(I_u(d))  = \frac{1}{1- \mathbb{E}(N)}.
\]

\end{enumerate}
\end{teo}

\begin{proof}
Items $(1)$ and $(2)$ are Propositions~\ref{P: CCSRSR2V} and \cite[Proposition 4.4]{MRV2018}, respectively. Item $(3)$ follows from (\ref{inequality}), Proposition~\ref{P: UlimPS} and \cite[Proposition 4.5]{MRV2018}. 

\end{proof}

\begin{cor}
 Consider $\procUp$ and $\procIp$. 
\begin{enumerate}
    \item Suppose that 
\[
\Phi \left( \frac{\lambda p }{\lambda p + 1}, 1,d+1\right ) > \frac{(\lambda p +1)[p(\lambda d + d + 1)-1]}{pd(d+1)(\lambda +1)}.
\]
Then 
\[
\mathbb{E}(I_u(d)) = 1 + \frac{(d+1)(\lambda +1) \left ( \lambda p + 1 - d\Phi \left ( \frac{\lambda p}{\lambda p +1}, 1, d \right ) \right )}{\lambda (\lambda p +1) \left ( 1 - \frac{d(\lambda +1)p}{(\lambda p+1)^2} \left [ \lambda p+1 -(d-1)\Phi \left ( \frac{\lambda p}{\lambda p +1}, 1, d \right ) \right ] \right )},
\]
\item Suppose that 
\[
p < \frac{d+1}{d +\lambda(d-1)}
\]
Then
\[
\mathbb{E}(I_i(d)) = 1 + \frac{(d+1)(\lambda +1)p(d + \lambda p)}{(d+\lambda p +1)(d(1- p(\lambda +1)) + \lambda p)}
\]
\item If $(\lambda+1)p<1,$
\[
\lim_{d \to \infty} \mathbb{E}(I_i(d)) = \lim_{d \to \infty} \mathbb{E}(I_u(d))  = \frac{1}{1- (\lambda+1)p}.
\]
\end{enumerate}
\end{cor}
\begin{proof}
    Follows from Theorem~\ref{th:coloniascriadas} and equations (\ref{distparticular}) and (\ref{fgpparticular}).
\end{proof}

\begin{teo}\label{extintiontime}  Let $\tau_u(d)$ and $\tau_i(d)$ be the extinction times  of the processes $\procUp$ and $\procIp$.

\begin{enumerate}
    \item Suppose that 
\[
\Phi \left( \frac{\lambda p }{\lambda p + 1}, 1,d+1\right ) > \frac{(\lambda p +1)[p(\lambda d + d + 1)-1]}{pd(d+1)(\lambda +1)}.
\]
Then
\[
\mathbb{E}(\tau_u(d)) = \int_{0}^{1} \frac{1-s}{G_u(s) - s}ds,
\]
where
\[
G_u(s) = \frac{1}{(\lambda p + 1)}\left [ 1-p + \frac{(\lambda +1)}{\lambda}\sum_{y=1}^{d} \left[ \dbinom{d}
{y}s^y \sum_{n=y}^{\infty} \left[ \left (\frac{\lambda p}{\lambda p +1}\right)^n \frac{\dbinom{n-1}
{y-1}}{\dbinom{n+d-1}{d-1}} \right] \right]  \right].
\]

\item  Suppose that 
\[
p < \frac{d+1}{d +\lambda(d-1)}
\]
Then
\[
\mathbb{E}(\tau_i(d)) = \int_{0}^{1} \frac{1-s}{G_i(s) - s}ds,
\]
where
\[
G_i(s) =   \frac{1}{\lambda p+1} \left[ 1- p  + \frac{\lambda + 1}{\lambda} \sum_{y=1}^{d}\left [ s^y \binom{d}{y} \sum_{n=y}^{\infty} T(n,y) \left (\frac{\lambda p}{d(\lambda p +1)} \right )^n \right ] \right],
\]
with \begin{displaymath}
T(n,k) = \sum_{i=0}^{k} \left [ (-1)^i \binom{k}{i}(k-i)^n \right], n \geq k.
\end{displaymath}
\end{enumerate}
\end{teo}
 
\begin{rem}
The quantity $T(n,k)$ gives the number of surjective functions $f:A \to B $, where $|A| = n$ and $|B| = k$, see  Tucker~\cite[p. 319]{Tucker}.
 \end{rem}

\begin{proof}
 Item $(1)$ is Proposition~\ref{ExtTimeU} with $\mathcal{N}=\mathcal{P}(\lambda)$ and $\mathcal{L}= \text{Bin}(p).$ Proof of item $(2)$ follows in an analogous way. On this last case, the probability generating function of the number of colonies created right after a catastrophe, $Y_d$, is given by 
 \[
G_i(s) =  \sum_{y=0}^{d}\left [ s^y \binom{d}{y} \sum_{n=y}^{\infty} T(n,y) \mathbb{P}[N=n] \right],
\]
see Machado et al~\cite[proof of Proposition 4.4]{MRV2018} for details. Finally, by using Equation~(\ref{distparticular}) follows the result. 
\end{proof}

\begin{exa} Consider $\procUp$ and $\procIp$ with 
 $\lambda =1$ and $p=\frac{1}{2}.$ By using Theorem~\ref{extintiontime} we compute the mean extinction times for both models with $2\leq d\leq 6$, see Table~\ref{tableET}. 
 \begin{table}[h]
 \centering
\caption{Mean extinction times for $\text{U}(\mathbb{T}^d_+; \mathcal{P}(1), \text{Bin}(\frac{1}{2}))$  and $\text{I}(\mathbb{T}^d_+; \mathcal{P}(1), \text{Bin}(\frac{1}{2})).$ }
\begin{tabular}{|c|c|c|c|c|c|}
  \hline

  $d$ & 2 & 3  & 4 &  5 & 6\\
  \hline
 $\mathbb{E}(\tau_u(d))$ & 3.494 & 3.862 & 4.159 & 4.408 & 4.623  \\
  \hline
  $\mathbb{E}(\tau_i(d))$  & 3.831 & 4.372 & 4.779 & 5.108 & 5.384 \\
  \hline
\end{tabular}\label{tableET}
\end{table}

\end{exa}

\end{document}